\documentclass[11pt]{article}
\usepackage{amsfonts,amssymb,amsthm,eucal,amsmath,verbatim}
\usepackage{setspace}

\usepackage[english]{babel}
\usepackage{tikz}
\usepackage[letterpaper,top=1in,bottom=1in,left=1in,right=1in,marginparwidth=1.75cm]{geometry}
\usepackage{subfigure}

\usepackage{mdframed}
\usepackage{graphicx}
\usepackage{mathtools}
\usepackage[colorlinks=true, allcolors=blue]{hyperref}
\usepackage{url} 
\usepackage{subfiles} 
\usepackage{amsfonts}
\usepackage{amsthm}
\usepackage{subcaption}
\usepackage[section]{placeins}
\usepackage{bbm} 
\usepackage{enumitem}
\usepackage{lmodern}
\usepackage{textcomp} 

\usepackage{mathtools}

\DeclarePairedDelimiterX{\norm}[1]{\lVert}{\rVert}{#1}

\newtheorem{theorem}{Theorem}[section]

\newtheorem{lemma}[theorem]{Lemma}

\newtheorem{proposition}[theorem]{Proposition}

\usepackage[normalem]{ulem}

\numberwithin{equation}{section}
\numberwithin{theorem}{section}

\usepackage{xcolor}

\definecolor{green}{rgb}{0.0, 0.5, 0.0}

\usepackage[normalem]{ulem}

\title{Rates of Bulk Convergence for Ensembles of Classical Compact Groups}
\author{Mengchun Cai%
\thanks{Department of Mathematics, Applied Mathematics, and
  Statistics, Case Western Reserve University, Cleveland, Ohio,
  U.S.A.; mengchun.cai@case.edu.}}
\date{}

\begin{document}
\maketitle
\begin{abstract}
This paper considers random matrices distributed according to the Haar measure in different classical compact groups. Utilizing the determinantal point structures of their nontrivial eigenangles, with respect to the $L^1$-Wasserstein distance, we obtain the rate of convergence for different ensembles towards the sine point process when the dimension of matrices $N$ is sufficiently large. Specifically, the rate is roughly of order $N^{-2}$ on the unitary group and of order $N^{-1}$ on the orthogonal group and the compact symplectic group.
\end{abstract}
\
\section{Introduction}
The eigenvalues of large random matrices drawn uniformly from classical compact groups are of interest in a variety of fields, including statistics, number theory, and mathematical physics. The nontrivial eigenangles of random matrices in $\mathbb{U}_N$, $\mathbb{SO}_N$ and $\mathbb{SP}_{2N}$  form  determinantal point processes converging to the sine point process after proper scaling when the dimension of matrices $N$ goes to infinity; see \cite{meckes_random_2019} for more details. With these limit laws established, one can ask for the rate, as a function of $N$, of the convergence.

Let us first introduce compact groups addressed in this paper. The {\bf orthogonal group}, $\mathbb{O}_N$, contains all $N\times N$ real matrices $U$ such that $U^TU=UU^T=I_N$. The {\bf unitary group}, $\mathbb{U}_N$, is formed by all $N\times N$ matrices $U$ over $\mathbb{C}$ with $U^*U=UU^*=I_N$. Similarly, any element in the compact {\bf symplectic group} $\mathbb{SP}_{2N}$ is a $2N\times 2N$ matrix $U\in\mathbb{U}(2N)$ such that $U^TJ_{N}U=UJ_NU^T=J_N$ where
\begin{equation*}
   J_N=\left( \begin{matrix}
        0&I_N\\
        -I_N&0
    \end{matrix}\right).
\end{equation*}
In particular, for $\mathbb{O}_N$, we tend to consider its two components: the {\bf special orthogonal group}, $\mathbb{SO}_N:=\left\{U\in\mathbb{O}_N:\det U=1\right\}$ and the {\bf negative coset}, which is not a group, $\mathbb{SO}^{-}_N:=\{U\in\mathbb{O}_N:\det U=-1\}$. According to Haar's theorem (see Theorem $1.14$ in \cite{meckes_random_2019} for our special setting), if $G$ is any group defined above, there is a unique probability measure $\mu$ on $G$ such that for any measurable subset $\mathcal{A}\subset G$ and fixed $M\in G$, we have
\begin{equation*}
    \mu\left(\mathcal{A}M\right)=\mu\left(M\mathcal{A}\right)=\mu(\mathcal{A}),
\end{equation*}
where $\mathcal{A}M=\{AM:A\in\mathcal{A}\}$ and $M\mathcal{A}=\{MA:A\in\mathcal{A}\}$. This $\mu$ is called the {\bf Haar probability measure} in $G$.

By some simple linear algebra, for any element in each group defined above, we see all of its eigenvalues must have absolute value $1$. Thus, it is reasonable for us to consider the corresponding angle. Given an eigenvalue $e^{i\theta},~0\le\theta<2\pi$, of a unitary matrix, we call $\theta$ an eigenangle of this matrix. To better capture properties induced by engenangles, especially the bulk result in this paper, we mainly focus on nontrivial eigenangles defined below. For each matrix in $\mathbb{SO}_{2N+1}$, we know it has $1$ as an eigenvalue, each matrix in $\mathbb{SO}^-_{2N+1}$ has $-1$ as an eigenvalue, and each matrix in $\mathbb{SO}^-_{2N+2}$ must have both $1$ and $-1$ as eigenvalues; those eigenvalues are referred as trivial eigenvalues.  When discussing $\mathbb{SO}_N,~\mathbb{SO}^-_N$ or $\mathbb{SP}_{2N}$, we say eigenangles corresponding to nontrivial eigenvalues in the upper half-circle are nontrivial eigenangles. This restriction is reasonable as all unreal eigenvalues occur in complex conjugate pairs. Furthormore, all eigenangles in $\mathbb{U}_N$ are considered nontrivial. By convention, in any of $G$ above, we refer the ensemble to the joint distribution of nontrivial eigenangles for the random matrix $U$, which is distributed according to the Haar probability measure $\mu$ in $G$.

A point process $\mathfrak{X}$ in a locally compact Polish space $\Omega$ is a random discrete subset of $\Omega$. For any $A\subset\Omega$, the integer-valued function:
\begin{equation*}
\mathcal{N}_\mathfrak{X}(A):=\#\{\omega\in\Omega:\omega\in A\}
\end{equation*}
is called the counting function of $\mathfrak{X}$. In this paper, we focus on point processes in $\mathbb{R}$. For convenience, for a random matrix $H_N$ on some compact group, $\mathcal{X}_{H_N}$ is referred to the process formed by the eigenangles of $H_N$ instead of eigenvalues. Similarly, given a measurable Borel $A\subset\mathbb{R}$, $\mathcal{N}_{H_N}(A)$ is defined to be equal to the number of corresponding eigenangles of $H_N$ in $A$. 

Fix a point process $\mathfrak{X}$. If they exist, the \textbf{correlation functions} (also called joint intensities) for $\mathfrak{X}$ are a sequence of locally integrable functions $\{\rho_k:\mathbb{R}^k\rightarrow\mathbb{R} \}_{k=1}^{\infty}$ satisfying the following condition. For all mutually disjoint subsets $\{D_j\}_{j=1}^n$ of $\mathbb{R}$ 
\begin{equation*}
	\mathbb{E} \left[\prod_{j=1}^{n} \mathcal{N}_{\mathfrak{X}}(D_j)\right] = \int_{\prod_j D_j} \rho_n(x_1,\dots,x_n) dx_1 \cdots dx_n,
\end{equation*}   
where the integral is with respect to the Lebesgue measure. Furthermore, if there exists a function $K(x,y)$ such that
\begin{equation*}
    \rho_n(x_1,\dots,x_n)=\det\left(K(x_i,x_j)_{i,j=1}^n\right)
\end{equation*}
for any $n$, then $\mathfrak{X}$ is called a \textbf{determinantal~point~process}, and $K$ is called the \textbf{kernel} of $\mathfrak{X}$. In particular, a determinantal point process $\mathfrak{X}_{Sine}$ given by the kernel:
\begin{equation*}
    K_{Sine}(x,y)=\left\{\begin{array}{cc}
      \frac{\sin(\pi(x-y))}{\pi(x-y)},~~x\ne y\\
      1,~~x=y
    \end{array}\right.
\end{equation*}
is called a \textbf{sine~point~process}.

For ensembles in compact groups defined above, we have the following proposition for their determinental point structures.
\begin{proposition}\label{kernel}
    {\rm\cite[Proposition $3.9$]{meckes_random_2019}} For any $N\in\mathbb{Z}^+$, let
    \begin{equation*}
        S_N(x):=\left\{\begin{array}{cl}
        \sin\left(\frac{Nx}{2}\right)/\sin\left(\frac{x}{2}\right),~~x\ne0,\\
        \\
        N,~~x=0.
        \end{array}\right.
    \end{equation*}
The nontrivial eigenangles of uniformly distributed random matrices in any of $\mathbb{SO}_{2N},~\mathbb{SO}^-_{2N},~\mathbb{U}_N$ and $\mathbb{SP}_{2N}$ are a determinantal point process, with respect to the Lebesgue measure on $\Lambda$, with the following kernels $K_N$.
\begin{center}
\begin{tabular}{c c c } 
 \hline
  & $K_N(x,y)$ & $\Lambda$ \\[0.5ex] \hline
 $\mathbb{U}_N$ & $\frac{1}{2\pi}S_N(x-y)$ & $[0,2\pi)$  \\[1ex] 
 \hline
 $\mathbb{SO}_{2N}$ & $\frac{1}{2\pi}\left(S_{2N-1}(x-y)+S_{2N-1}(x+y)\right)$ & $[0,\pi)$   \\[1ex]
 \hline
 $\mathbb{SO}_{2N+1}$ &$\frac{1}{2\pi}\left(S_{2N}(x-y)-S_{2N}(x+y)\right)$ & $[0,\pi)$   \\[1ex]
 \hline
 $\mathbb{SO}^-_{2N+1}$ &$\frac{1}{2\pi}\left(S_{2N}(x-y)+S_{2N}(x+y)\right)$ & $[0,\pi)$   \\[1ex]
 \hline
 $\mathbb{SP}_{2N},~\mathbb{SO}^-_{2N+2}$ & $\frac{1}{2\pi}\left(S_{2N+1}(x-y)-S_{2N+1}(x+y)\right)$ & $[0,\pi)$   \\[1ex]
 \hline
\end{tabular}
\end{center}
\end{proposition}
\bigskip

Given a Polish metric space $(\Omega,d)$ and two probability measures $\mu$ and $\nu$ on $\Omega$, the $L^1$-Wasserstein distance $W_1$ between $\mu$ and $\nu$ is defined as:
\begin{equation*}
    W_1(\mu,\nu):=\inf_{(x,y)\in\pi(\mu,\nu)}\mathbb{E}\left(d(x,y)\right)
\end{equation*}
where $\pi(\mu,\nu)$ contains all couplings of $\mu$ and $\nu$.

As mentioned in the first paragraph, the sine point process is the limiting process for all ensembles addressed in this paper, and our goal is to provide a rate for this asymptotic behavior. To be specific, in this paper, we quantify the convergence rate with respect to $W_1$, for ensembles in different compact groups when the dimension $N$ is sufficiently large. The method of this paper relies on the decomposition of operators, which is similar to the philosophy in \cite{cai2025microscopicratesconvergencehermitian}, the author's joint paper with Dr. Kyle. For the unitary ensemble in $\mathbb{U}_N$, which is also called a \textbf{ CUE~(Circluar~Unitary~Ensemble)}, we will show that the rate is essentially of order $N^{-2}$. For other ensembles, we get $N^{-1}$ as the convergence rate. We conjecture that all of these rates are optimal.

\section{Statement of the results}
Following results will be shown in this paper.
\begin{theorem}\label{C bulk}
    Given a CUE $C_N$, let $I=[-s,s]$ for $s>0$ be an interval and $\mathcal{X}_{C_N}^{Bulk}$ be the bulk-scaled eigenangle process with the kernel $K_{C_N}^{Bulk}(x,y):=\frac{2\pi}{N}K_{C_N}(\frac{2\pi x}{N}+\pi,\frac{2\pi y}{N}+\pi)$. Consider two point processes $\mathcal{X}_{C_N}^{Bulk}$ and $\mathcal{X}_{Sine}$ restricted on $A\subset I$. Under the condition that $\frac{2s}{N}<1$, for $N$ sufficiently large, 
\begin{equation}\label{main CUE}
W_1(\mathcal{N}_{C_N}^{Bulk},\mathcal{N}_{Sine})\le \frac{CN^2s^2}{N^4-16s^4}
\end{equation}
where $C$ is some uniform constant.
\end{theorem}

\begin{theorem}\label{C2}
  Let $H_N$ be a random matrix distributed according to the Haar probability measure in any of $\mathbb{SO}_{2N},~\mathbb{SO}_{2N+1},~\mathbb{SO}^{-}_{2N},~\mathbb{SO}^-_{2N+1}$ and $\mathbb{SP}_{2N}$, $I=[-s,s]$ for $s>0$ be an interval.  Consider two point processes $\mathcal{X}_{H_N}^{Bulk}$ and $\mathcal{X}_{Sine}$ on $A\subset I$ with kernels $K_{H_N}^{Bulk}(x,y):=\frac{\pi}{N}K_{H_N}(\frac{\pi x}{N}+\frac{\pi}{2},\frac{\pi y}{N}+\frac{\pi}{2})$ and $K_{Sine}$ respectively. Then for $N$ sufficiently large,
\begin{equation}\label{main He}
W_1(\mathcal{N}_{H_N}^{Bulk},\mathcal{N}_{Sine})\le \frac{Cs}{N}
\end{equation}
for some uniform constant $C$ under the assumption that $\frac{2s}{N}<1$.
\end{theorem}
The notation has been abbreviated as $\mathcal{N}:=\mathcal{N}(A)$ in those two theorems above, but keep in mind that all point processes above are on $A$.
    In addition, the condition $\frac{2s}{N}<1$ in both theorems can be regarded as the bulk restriction. Notice that after our recentering, all eigenangles are located in $[-\frac{N}{2},\frac{N}{2}]$. This restriction means that our interval $I$ lies away from two edges and hence, our focus is on those points around the bulk. In addition, for a fixed $s$, we have $N^{-2}$ in \eqref{main CUE} and $N^{-1}$ in \eqref{main He} as main orders of the convergence rates. Furthermore, if $s=o\left(N^{\frac{1}{4}}\right)$, then Theorem \ref{C bulk} in fact improves the order $N^{-\frac{3}{2}}$ in \cite[Corollary $6$]{meckes2015self}.
\section{Outline of the Proof}
\subsection{Decomposition of Operators}
For a kernel function $K(x,y)$ on some domain $D^2\subset\mathbb{R}^2$, we can define a corresponding integral operator $\mathcal{K}$ by the formula:
\begin{equation*}
    \mathcal{K}(f)(x)=\int_D K(x,y)f(y)dy
\end{equation*}
for some suitable function $f$ defined on $D$. In addition, as a result of the discussions in \cite{soshnikov2000determinantal}, if $K$ is the kernel function of any determinantal point process mentioned above, then the corresponding operator is a locally trace class operator:
\begin{equation*}
    \|\mathcal{K}\|_{1,U}=\sum_{j}s_j(\mathcal{K}|_U)<\infty
\end{equation*}
where $U$ is a bounded Borel set on $\mathbb{R}$, $\mathcal{K}|_U$ is the restriction of $\mathcal{K}$ on $U$ and $s_j$'s are singular values. Moreover, define the \textbf{Hilbert-Schmidt~norm} $\|\mathcal{K}\|_{2,U}$ as:
\begin{equation}\label{2 to L^2}
    \|\mathcal{K}\|_{2,U}:=\left(\sum_{j}s_j^2(\mathcal{K}|_U)\right)^{\frac{1}{2}}=\left(\int_U\int_U |K|^2(x,y)dydx\right)^{\frac{1}{2}}=\|K\|_{L^2(U^2)}.
\end{equation}
We call $\mathcal{K}$ a Hilbert Schmidt operator if $\|\mathcal{K}\|_2<\infty$. Given Hilbert-Schmidt operators $\mathcal{K}_1$ and $\mathcal{K}_2$ with $K_1$ and $K_2$ as corresponding kernels, we have the \textbf{operator~Cauchy-Schwarz~inequality}:
\begin{equation}\label{CS in}
\|\mathcal{K}_1\mathcal{K}_2\|_1\le\|\mathcal{K}_1\|_2\|\mathcal{K}_2\|_2=\|K_1\|_{L^2}\|K_2\|_{L^2}.
\end{equation}

With the help of the following lemma, to bound the $W_1$ distance between determinantal point processes, we only need to control the trace class norm between associated integral operators.  
\begin{lemma}
	\label{M and M}{\rm\cite[Lemma $1.1$]{cai2025microscopicratesconvergencehermitian}}
	Consider two determinantal point processes $\mathfrak{X}$ and $\widetilde{\mathfrak{X}}$ with Hermitian kernels $K(x,y)$ and $\widetilde{K}(x,y)$ and associated integral operators $\mathcal{K}$ and $\widetilde{\mathcal{K}}$. Assume the integral operators are trace class. Then,
	\begin{equation*}
	d_{TV}(\mathcal{N}_{\mathcal{X}},\mathcal{N}_{\tilde{\mathcal{X}}}) 
	\leq W_1(\mathcal{N}_\mathcal{X},\mathcal{N}_{\tilde{\mathcal{X}}})
	\leq \lVert \mathcal{K} - \widetilde{\mathcal{K}} \rVert_1.
	\end{equation*}
\end{lemma}
In this paper, our main strategy is to control the trace class norm $\|\cdot\|_1$ of the difference between two operators. The Cauchy-Schwarz inequality \eqref{CS in} is used in our proof to overcome the difficulty of the direct computation of the trace class norm, since the Hilbert-Schmidt norm $\|\cdot\|_2$ of an integral operator is equal to the $L^2$-norm of its corresponding kernel as in \eqref{2 to L^2}. 

Let $s>0$ and $I=[-s,s]$ be some interval. For any $j\in\mathbb{N}$, define functions:
\begin{equation*}
    C^\pm_{j}(x)=\cos\left(\pi x\right)\left(\pm\pi x\right)^j~~and~~ S^\pm_{j}(x)=\sin\left(\pi x\right)\left(\pm\pi x\right)^j
\end{equation*}
on $I$. Furthermore, let $\mathcal{C}^\pm_{j}$ and $\mathcal{S}^\pm_{j}$ be the corresponding integral operators on $L^2(I)$ with kernel $C^\pm_{j}(x,y):=C^\pm_{j}(y)$ and $S^\pm_{j}(x,y):=S^\pm_{j}(y)$ respectively. Similarly to the method in \cite{el2006rate} and \cite{johnstone2012fast}, the following decompositions play an important role in our discussion.
\begin{proposition}\label{CUE fact}
Let $I=[-s,s]$ for some $s>0$, $\mathcal{A}_{2k+1}$ and $\mathcal{A}'_{2k+1}$ be operators on $L^2(I)$ with kernels:
\begin{equation*}
A_{2k+1}(x,y)=(\pi(x-y))^{2k+1}\sin(\pi(x-y))
\end{equation*}
and
\begin{equation*}
A'_{2k+1}(x,y)=(\pi(x+y))^{2k+1}\sin(\pi(x+y))
\end{equation*}
respectively.
We have
\begin{equation*}
\mathcal{A}_{2k+1}=\sum_{j=0}^{2k+1}\left(\begin{matrix}
    2k+1\\
    j
\end{matrix}\right)M_{S^+_{j}}\mathcal{C}^-_{2k+1-j}-\sum_{j=0}^{2k+1}\left(\begin{matrix}
    2k+1\\
    j
\end{matrix}\right)M_{C^+_{j}}\mathcal{S}^-_{2k+1-j}
\end{equation*}
and
\begin{equation*}
\mathcal{A}'_{2k+1}=\sum_{j=0}^{2k+1}\left(\begin{matrix}
    2k+1\\
    j
\end{matrix}\right)M_{S^+_{j}}\mathcal{C}^+_{2k+1-j}+\sum_{j=0}^{2k+1}\left(\begin{matrix}
    2k+1\\
    j
\end{matrix}\right)M_{C^+_{j}}\mathcal{S}^+_{2k+1-j}
\end{equation*}
where $M_f:g\rightarrow f\cdot g$ is the multiplication operator.
\end{proposition}
\begin{proof}
Directly applying the binomial formula and trigonometric identities to $A_{2k+1}$ leads to
\begin{equation}\label{S4A}
    \begin{split}
        &A_{2k+1}(x,y)=\sum_{j=0}^{2k+1}\left(\begin{matrix}
            2k+1\\j
        \end{matrix}\right)(\pi x)^j(-\pi y)^{2k+1-j}\sin(\pi x)\cos(\pi y)\\&~~~~~~~~~~~~~-\sum_{j=0}^{2k+1}\left(\begin{matrix}
            2k+1\\j
        \end{matrix}\right)(\pi x)^j(-\pi y)^{2k+1-j}\sin(\pi y)\cos(\pi x)\\
        &=\sum_{j=0}^{2k+1}\left(\begin{matrix}
            2k+1\\j
        \end{matrix}\right)S_{j}^+(x)C_{2k+1-j}^{-}(y)-\sum_{j=0}^{2k+1}\left(\begin{matrix}
            2k+1\\j
        \end{matrix}\right)C_{j}^+(x)S_{2k+1-j}^-(y).
    \end{split}
\end{equation}
For each $j$, assume $\tilde{K}_j(x,y)=S_j^+(x)C_{2k+1-j}^-(y)$ is the integral kernel of $\tilde{\mathcal{K}}_{j}$ on $L^2(I)$. Then for any $f\in L^2(I)$, 
\begin{equation*}
    \tilde{\mathcal{K}}_jf(x)=\int_I S_j^+(x)C_{2k+1-j}^-(y)f(y)dy=M_{S_j^+}\left(\int_I C_{2k+1-j}^-(y)f(y)dy\right)=M_{S_j^+}\mathcal{C}_{2k+1-j}^-f(x)
\end{equation*}
and hence, $\tilde{\mathcal{K}}_j=M_{S_j^+}\mathcal{C}_{2k+1-j}^-$. Similar discussion for $C^+_j(x)S_{2k+1-j}^-(y)$, together with the linearity of the integral, yields the decomposition for $\mathcal{A}_{2k+1}$ from \eqref{S4A}.

Substituting all negative signs by the positive ones in the previous discussion for $A_{2k+1}(x,y)$ completes the proof for $\mathcal{A}'_{2k+1}$.
\end{proof}

\begin{proposition}\label{fac2}
    Let $I=[-s,s]$ for some $s>0$ and $K_2(x,y):=\cos\left(\pi(x-y)\right),~K_3(x,y):=\cos(\pi(x+y))$ be kernel functions on $I^2$. The following decomposition holds for their corresponding integral operators:
\begin{equation*}
\mathcal{K}_2=M_{C^+_0}\mathcal{C}^+_0+M_{S^+_0}\mathcal{S}^+_0,~~
\mathcal{K}_3=M_{C^+_0}\mathcal{C}^+_0-M_{S^+_0}\mathcal{S}^+_0.
\end{equation*}
\end{proposition}
\begin{proof}
    Apply the trigonometric identities to $K_2$ and $K_3$, then argue as in the previous proof.
\end{proof}

The following lemma establishes a bound for each term of the decomposition in Proposition \ref{CUE fact}.
\begin{lemma}\label{b4cs}
    For any $j,~k\in\mathbb{N}$, let $I=[-s,s]$ for some $s>0$, $f_j\in\{C_j^+,C_j^-,S_j^+,S_j^-\}$ and $\mathcal{H}_k\in\{\mathcal{C}_k^+,\mathcal{C}_k^-,\mathcal{S}_k^+,\mathcal{S}_k^-\}$. As an operator on $L^2(I)$, $M_{f_j}\mathcal{H}_k$ is with rank $1$ such that
    \begin{equation*}
        \|M_{f_j}\mathcal{H}_k\|_1=\|f_jh_k\|_{L^2(I^2)}\le\frac{2s(\pi s)^{j+k}}{\sqrt{(2j+1)(2k+1)}}
    \end{equation*}
    where $h_k$ is the corresponding kernel.
\end{lemma}
\begin{proof}
For any $j,~k\in\mathbb{N}$, the multiplication operator $M_{f_j}$ is always a rank $1$ operator. In addition, since the corresponding kernel $h_k$ is a function of single variable, as an integral kernel, $\mathcal{H}_k$ is also with rank $1$. Consequently, the composition operator $M_{f_j}\mathcal{H}_k$ must be a rank $1$ operator and hence, has one unique singular value $s_{jk}$. Recall the kernel of $M_{f_j}\mathcal{H}_k$ is $f_j(x)h_k(y)$, it follows that
\begin{equation*}
    \begin{split}
\|M_{f_j}\mathcal{H}_k\|^2_1&=s^2_{jk}=\|M_{f_j}\mathcal{H}_k\|_2^2=\|f_j(x)h_k(y)\|^2_{L^2(I^2)}\\&=\int_I|f_j(x)|^2dx\int_I|h_k(y)|^2dy\\
&\le\int_{-s}^s(\pi x)^{2j}dx\int_{-s}^s(\pi y)^{2k}dy\\
&=\frac{4s^2(\pi s)^{2j+2k}}{(2j+1)(2k+1)}
    \end{split}
\end{equation*}
from \eqref{2 to L^2}. Taking the square root yields the promised result.
\end{proof}

\subsection{CUE bulk}\label{CUE bulk}
The proof of the CUE bulk result depends on the Laurent expansion of the cosecant function $\csc$ after comparison between $K_{C_N}^{Bulk}$ and $K_{Sine}$. For $0<|x|<\pi$, recall the Laurent series:
\begin{equation}\label{csc}
    \csc(x)=\sum_{n=0}^\infty\frac{(-1)^{n+1}2\left(2^{2n-1}-1\right)B_{2n}x^{2n-1}}{(2n)!}
\end{equation}
where $B_k$ is the $k$th Bernoulli number. Before going into the main proof, we first introduce one lemma to control the main constant terms appearing in our discussion below. 
\begin{lemma}\label{Bell}
    For any $n\in\mathbb{Z}^+$, we have $|B_{2n}|\le Cn^{2n+\frac{1}{2}}(\pi e)^{2n}$ for some uniform constant $C$ where $B_k$ is the $k$th Bernoulli number.
\end{lemma}
\begin{proof} For any $n\in\mathbb{Z}^+$, according to \cite[Corollary $4.12$]{arakawa_bernoulli_2014}, the Bernoulli number $B_{2n}$ can be expressed as:
    \begin{equation*}
        B_{2n}=\frac{(-1)^{n+1}2(2n)!}{(2\pi)^{2n}}\zeta(2n)
    \end{equation*}
    where $\zeta$ is the Riemann zeta function. By the definition of $\zeta$, it is obvious that $1\le\zeta(t)\le\zeta(2)=\frac{\pi^2}{6}$ for any $t\in[2,+\infty)$. Thus, the dominated convergence theorem derives
\begin{equation*}
\lim_{t\in\mathbb{R},~t\rightarrow+\infty}\zeta(t)=\sum_{n=1}^\infty\lim_{t\in\mathbb{R},~t\rightarrow+\infty}\frac{1}{n^t}=1
\end{equation*}
and hence, $\lim\limits_{n\rightarrow+\infty}\zeta(2n)=1$.
Applying Stirling's formula: $n!e^n\sim\sqrt{2\pi}n^{n+\frac{1}{2}}$, it is not difficult to see
    \begin{equation*}
    \lim_{n\rightarrow+\infty}\frac{(\pi e)^{2n}|B_{2n}|}{n^{2n+\frac{1}{2}}}=4\sqrt\pi.
    \end{equation*}
Consequently, there exists a uniform constant $C>4\sqrt\pi$ such that $|B_{2n}|\le C n^{2n+\frac{1}{2}}(\pi e)^{2n}$ for any $n$.
\end{proof}

\medbreak\noindent {\itshape Proof of Theorem \ref{C bulk}.}\enspace Recall
\begin{equation*}
K_{C_N}^{Bulk}(x,y)=\frac{\sin(\pi(x-y))}{N\sin\left(\frac{\pi(x-y)}{N}\right)}
\end{equation*}
and
\begin{equation*}
K_{Sine}(x,y)=\frac{\sin(\pi(x-y))}{\pi(x-y)}.
\end{equation*}
On any compact domain $D^2\subset\mathbb{R}^2$, according to \eqref{csc}, we can write the difference as:
\begin{equation*}
K_{C_N}^{Bulk}(x,y)-K_{Sine}(x,y)=\frac{1}{N^2}\sum_{k=0}^\infty\frac{c_{2k+1}}{N^{2k}}(\pi(x-y))^{2k+1}\sin(\pi(x-y))=\frac{1}{N^2}\sum_{k=0}^\infty\frac{c_{2k+1}}{N^{2k}}A_{2k+1}(x,y)
\end{equation*}
where
\begin{equation*}
    c_{2k+1}=\frac{(-1)^k2(2^{2k+1}-1)B_{2k+2}}{(2k+2)!}.
\end{equation*}
Define an operator
\begin{equation*}
\mathcal{K}^d_N:=\frac{1}{N^2}\sum_{k=0}^\infty\frac{c_{2k+1}}{N^{2k}}\mathcal{A}_{2k+1}
\end{equation*}
on $L^2(D^2)$.
It is clear that if the right side above is convergent with respect to $\|\cdot\|_1$, then we must have $\mathcal{K}_{C_N}^{Bulk}-\mathcal{K}_{Sine}=\mathcal{K}_N^d$. In fact, 
Proposition \ref{CUE fact}, together with Lemma \ref{b4cs}, yields:
\begin{equation}\label{operator A}
\begin{split}
    &\|\mathcal{A}_{2k+1}\|_1=\sum_{j=0}^{2k+1}\left(\begin{matrix}
    2k+1\\
    j
\end{matrix}\right)\|M_{S^+_{j}}\mathcal{C}^-_{2k+1-j}\|_1+\sum_{j=0}^{2k+1}\left(\begin{matrix}
    2k+1\\
    j
\end{matrix}\right)\|M_{C^+_{j}}\mathcal{S}^-_{2k+1-j}\|_1\\&\le2\sum_{j=0}^{2k+1}\left(\begin{matrix}
    2k+1\\
    j
\end{matrix}\right)\frac{2s(s\pi)^{2k+1}}{\sqrt{(2j+1)(4k+4-(2j+1))}}\\&
\le\frac{4s(s\pi)^{2k+1}}{\sqrt{4k+3}}\sum_{j=0}^{2k+1}\left(\begin{matrix}
    2k+1\\j
\end{matrix}\right)=\frac{4s(2\pi s)^{2k+1}}{\sqrt{4k+3}}
\end{split}
\end{equation}
where the fact that $(2j+1)(4k+4-(2j+1))\ge 4k+3$ for any $0\le j\le 2k+1$ has been used in the third inequality and the binomial expression: $2^m=(1+1)^m$ in the last one. Establishing a constant $C_1$ such that the far right term in the last line of \eqref{operator A} is bounded by $C_1s(2\pi s)^{2k+1}$ for any $k$ is not difficult, given its fraction term with respect to $k$.

Now go back to our original operator; we see
\begin{equation}\label{series}
    \|\mathcal{K}_N^d\|_1\le\frac{C_1s}{N^2}\sum_{k=0}^\infty\frac{|c_{2k+1}|(2\pi s)^{2k+1}}{N^{2k}}\le\frac{C_1s}{N^2}\sum_{k=0}^\infty\frac{2^{2k+2}|B_{2k+2}|(2\pi s)^{2k+1}}{(2k+2)!N^{2k}}.
\end{equation}
For the numerator, there exists a uniform constant $C_2$ such that $|B_{2k}|\le C_2k^{2k+\frac{1}{2}}(\pi e)^{-2k}$ for any $k$ by Lemma \ref{Bell}. The denominator can be controlled by Stirling's formula-there is a constant $C_3$ satisfying $k!\ge C_3\sqrt{k}(k)^ke^{-k}$ for any $k$. Thus, each term in the summation of \eqref{series} is bounded as:
\begin{equation*}
        \frac{2^{2k+2}(2\pi s)^{2k+1}|B_{2k+2}|}{(2k+2)!N^{2k}}\le\frac{(2\pi s)^{2k+1}C_22^{2k+2}(k+1)^{2k+2}\sqrt{k+1}e^{2k+2}}{(\pi e)^{2k+2}C_3\sqrt{2k+2}(2k+2)^{2k+2}N^{2k}}=\frac{\sqrt2C_2s}{\pi C_3}\left(\frac{2s}{N}\right)^{2k}.
\end{equation*}
Recall $\frac{2s}{N}<1$ by our assumption, the trace class norm of $\mathcal{K}_N^d$ is in fact dominated by a convergent geometric sequence as:
\begin{equation*}
\|\mathcal{K}_N^d\|_1\le\frac{Cs^2}{N^2}\sum_{k=0}^\infty\left(\frac{2s}{N}\right)^k=\frac{Cs^2}{N^2}\frac{1}{1-\frac{16s^4}{N^4}}=\frac{CN^2s^2}{N^4-16s^4}
\end{equation*}
and hence, $\|\mathcal{K}_{C_N}^{Bulk}-\mathcal{K}_{Sine}\|_1$ is also bounded by the same term on the far right.\qed

\subsection{Other ensembles}\label{other bulk}
\medbreak\noindent {\itshape Proof of Theorem \ref{C2}.}\enspace We first consider the case of $\mathbb{SO}_{2N}$. Split the corresponding kernel as:
\begin{equation*}
\begin{split}
K_{H_N}^{Bulk}(x,y)
&=\frac{\sin\left(\left(1-\frac{1}{2N}\right)\pi(x-y)\right)}{2N\sin\left(\frac{\pi(x-y)}{2N}\right)}-\frac{\cos\left(\left(1-\frac{1}{2N}\right)\pi(x+y)+N\pi\right)}{2N\cos\left(\frac{\pi(x+y)}{2N}\right)}\\
&=\frac{\sin\left(\pi(x-y)\right)\cot\left(\frac{\pi(x-y)}{2N}\right)}{2N}-\frac{\cos\left(\pi(x-y)\right)}{2N}\\&~~-\frac{(-1)^N\cos\left(\pi(x+y)\right)}{2N}-\frac{(-1)^N\sin\left(\pi(x+y)\right)\tan\left(\frac{\pi(x+y)}{2N}\right)}{2N}\\&:=K_{1,N}(x,y)-K_{2,N}(x,y)-(-1)^NK_{3,N}(x,y)-(-1)^NK_{4,N}(x,y).
\end{split}
\end{equation*}
Observe $K_{1,N}$ converges to $ K_{Sine}$ as $N$ goes to infinity on $A^2\subset I^2$, it follows that the decomposition of the corresponding operator $\mathcal{K}_{1,N}-\mathcal{K}_{Sine}$ as in the proof of Theorem \ref{C bulk} must be the key in our argument. Using the Laurent expansion for the cotangent function $\cot$ now,
\begin{equation*}
K_{1,N}(x,y)-K_{Sine}(x,y)=\frac{1}{N^2}\sum_{k=0}^\infty\frac{b_{2k+1}}{N^{2k}}(\pi(x-y))^{2k-1}\sin(\pi(x-y))
\end{equation*}
where
\begin{equation*}
    b_{2k+1}=\frac{(-1)^{k+1}B_{2k+2}}{(2k+2)!}
\end{equation*}
is the Laurent coefficient for the $\cot$ function. The same as the previous proof, now we hope to bound the trace class norm of
\begin{equation*}
\frac{1}{N^2}\sum_{k=0}^\infty\frac{b_{2k+1}}{N^{2k}}\mathcal{A}_{2k-1}.
\end{equation*}
 For each term in this summation, according to \eqref{operator A},
\begin{equation*}
    \|\mathcal{A}_{2k-1}\|_1\le\frac{4s(2\pi s)^{2k-1}}{\sqrt{4k-1}}.
\end{equation*}
Hence, there is constant $C_1$ such that the right side above is bounded by $C_1s(2\pi s)^{2k-1}$ for any $k$.

Going back to the difference operator, it follows that
\begin{equation}\label{series2}
    \left\|\frac{1}{N^2}\sum_{k=0}^\infty\frac{b_{2k+1}}{N^{2k}}\mathcal{A}_{2k-1}\right\|_1\le\frac{C_1s}{N^2}\sum_{k=0}^\infty\frac{|B_{2k+2}|(2\pi s)^{2k-1}}{(2k+2)!N^{2k}}.
\end{equation}
By Lemma \ref{Bell} and Stirling's formula, each term in the right summation of $(\ref{series2})$ is bounded by $\frac{C_2(2s)^{2k-1}}{N^{2k}}$ for some constant $C_2$. Thus, according to similar discussion to the previous proof, the inequality
\begin{equation*}
 \left\|\frac{1}{N^2}\sum_{k=0}^\infty\frac{b_{2k+1}}{N^{2k}}\mathcal{A}_{2k-1}\right\|_1\le\frac{Cs}{N^2}\frac{1}{1-\frac{16s^4}{N^4}}=\frac{CN^2s}{N^4-16s^4}
\end{equation*}
holds for some uniform constant $C$.

The remaining kernels can be easily handled. The bound $\frac{2s}{N}$ for both $\|\mathcal{K}_{2,N}\|_1$ and $\|\mathcal{K}_{3,N}\|_1$ can be directly derived by the decomposition in Proposition \ref{fac2} and Lemma \ref{b4cs}.
For the last kernel $K_4$, directly using Taylor expansion for the tangent function leads to
\begin{equation*}
    K_{4,N}(x,y)=\frac{1}{2N}\sum_{k=0}^\infty a_{2k+1}\left(\frac{x+y}{2N}\right)^{2k+1}\sin\left(\pi(x+y)\right)=\sum_{k=0}^\infty\frac{a_{2k+1}}{(2N)^{2k+2}}A'_{2k+1}(x,y)
\end{equation*}
where 
\begin{equation*}
    a_{2k+1}=\frac{(-4)^{k+1}(1-4^{k+1})B_{2k+2}}{(2k+2)!}
\end{equation*}
are the Taylor coefficients for the $\tan$ function. As above, applying Lemma \ref{Bell} and Stirling's formula to the operator
    $\sum\limits_{k=0}^\infty\frac{a_{2k+1}}{(2N)^{2k+2}}\mathcal{A}'_{2k+1}$
yields:
\begin{equation*}
    \left\|\sum_{k=0}^\infty\frac{a_{2k+1}}{(2N)^{2k+2}}\mathcal{A}'_{2k+1}\right\|_1=\frac{1}{2N}\left\|\sum_{k=0}^\infty\frac{a_{2k+1}}{(2N)^{2k+1}}\mathcal{A}'_{2k+1}\right\|_1\le\frac{CNs}{N^4-16s^4}
\end{equation*}
for some constants $C$. It is because the inequality $\|\mathcal{A}'_{2k+1}\|_1\le(C_1s)(2\pi s)^{2k+1}$ can be derived from the the same discussion as for $\mathcal{A}_{2k+1}$ above.

Putting all the bounds above together, we see that the main order $N^{-1}$ is obtained by $\mathcal{K}_{2,N}$ and $\mathcal{K}_{3,N}$. Absorbing all the constants, we finally get
\begin{equation*}
    \|\mathcal{K}_{H_N}^{Bulk}-\mathcal{K}^{Sine}\|_1\le\frac{Cs}{N}.
\end{equation*}
The analogous result for other ensembles is obtained by a straightforward modification of the sign conventions in our analysis of $\mathbb{SO}_{2N}$, leveraging the structural similarities of their kernels outlined in Proposition \ref{kernel}.
\qed

\bigskip
\noindent{\bf Acknowledgements. }The author would like to thank Doctor Kyle Taljan for the inspiring ideas in his PhD thesis and Professor Mark Meckes for his comments and valuable feedback. The author is also grateful to an anonymous referee for the suggestion that helps improve the final convergence rates in this paper.

\newpage

\bibliographystyle{plain}
\bibliography{my_references}

\end{document}